\documentclass[preprint,1p,10pt]{elsarticle}

\journal{Arxiv}

\biboptions{sort&compress,comma,round}
\usepackage{stmaryrd}
\usepackage{graphicx}
\usepackage{amssymb}
\usepackage{bm}
\usepackage[FIGTOPCAP]{subfigure}
\usepackage{amsmath}
\usepackage{float}
\DeclareSymbolFont{largesymbolsA}{U}{txexa}{m}{n}
\usepackage{setspace}
\usepackage{graphics}
\usepackage{amssymb, latexsym, amsmath, epsfig}
\usepackage{graphicx}
\usepackage{epstopdf}
\usepackage{comment}
\usepackage[normalem]{ulem}
\usepackage{float}
\usepackage{amssymb,amsmath}
\usepackage{amsfonts,amstext,amsthm,amssymb,comment}
\usepackage{epsfig,graphics,color}

\graphicspath{{figures/}}

\newtheorem{theorem}{Theorem}

\numberwithin{equation}{section}

\newtheorem{remark}{Remark}

\newcommand{\tr}{\text{tr }}

\newcommand{\del}[1]{\llbracket #1 \rrbracket}

\newcommand{\bfs}[1]{{\boldsymbol #1}}

\newcommand{\Question}[1]{\marginpar{}}
\renewcommand{\Question}[1]{%
  \marginpar{\flushleft\scriptsize\bfseries\upshape#1}}

\begin{document}
	
\begin{frontmatter}
		
		
  \title{Higher-order generalized-$\alpha$ methods for parabolic problems}
  
  \author[ad1]{Pouria Behnoudfar\corref{corr}}
  \ead{pouria.behnoudfar@postgrad.curtin.edu.au}
  \cortext[corr]{Corresponding author}
  \author[ad2]{Quanling Deng}
  \ead{Quanling.Deng@math.wisc.edu}
  \author[ad3]{Victor M. Calo}
  \ead{Victor.Calo@curtin.edu.au}
  \address[ad1]{Curtin Institute for Computation \& School of Earth and Planetary Sciences, Curtin University, Kent Street, Bentley, Perth, WA 6102, Australia}
  \address[ad2]{Department of Mathematics, University of Wisconsin-Madison, Madison, WI 53706, USA.}
  \address[ad3]{Curtin Institute for Computation \& School of Electrical
    Engineering, Computing and Mathematical Sciences, Curtin
    University, P.O. Box U1987, Perth, WA 6845, Australia}

  \begin{abstract}
    We propose a new class of high-order time-marching schemes with dissipation user-control and unconditional stability for parabolic equations. High-order time integrators can deliver the optimal performance of highly-accurate and robust spatial discretizations such as isogeometric analysis. The generalized-$\alpha$ method delivers unconditional stability and second-order accuracy in time and controls the numerical dissipation in the discrete spectrum's high-frequency region. Our goal is to extend the generalized-$alpha$ methodology to obtain a high-order time marching methods with high accuracy and dissipation in the discrete high-frequency range. Furthermore, we maintain the stability region of the original, second-order generalized-$alpha$ method foe the new higher-order methods. That is, we increase the accuracy of the generalized-$\alpha$ method while keeping the unconditional stability and user-control features on the high-frequency numerical dissipation. The methodology solve $k>1, k\in \mathbb{N}$ matrix problems and updates the system unknowns, which correspond to higher-order terms in Taylor expansions to obtain $(3/2k)^{th}$-order method for even $k$ and $(3/2k+1/2)^{th}$-order for odd $k$. A single parameter $\rho^\infty$ controls the dissipation, and the update procedure follows the formulation of the original second-order method.  Additionally, we show that our method is A-stable and setting $\rho^\infty=0$ allows us to obtain an L-stable method. Lastly, we extend this strategy to analyze the accuracy order of a generic method. 
  \end{abstract}
		
  \begin{keyword}
    generalized-$\alpha$ method \sep spectrum analysis \sep parabolic equation \sep dissipation control \sep stability analysis
			
  \end{keyword}
		
\end{frontmatter}
	
\vspace{0.5cm}
	
\section{Introduction}
\citet{ jansen2000generalized} introduced the generalized-$\alpha$ method for parabolic problems as a time-marching scheme, with second-order accuracy and unconditional stability. Furthermore, the methodology allows the user to control the numerical dissipation in the high-frequency region. That is, the generalized-$\alpha$ method controls the high- and low-frequency dissipations in the sense that for a given high-frequency dissipation, the algorithm minimizes the low-frequency dissipation. Thus, one obtains accurate approximations in both low- and high- frequency regions; see~\cite{ jansen2000generalized, chung1993time}. Despite these features, to date, the generalized-$\alpha$ method is limited to second-order accuracy in time while the high-order Runge-Kutta and multistep schemes (e.g., Adams-Moulton, and backward differentiation formulae (BDF)) lack explicit control over the numerical dissipation of the high frequencies (see~\cite{ butcher2016numerical, ascher1997implicit, burbeau2001problem}). (Remarkably, BDF2 corresponds to the generalized-$alpha$ method with maximal high-frequency dissipation, $\rho_\infty=0$.) Another shortcoming of these high-order multistep methods is that their stability regions shrink as their order increases. Therefore, unconditional stability is not possible for higher orders than two. While Runge-Kutta methods show better stability regions and deliver A-stability with higher-order accuracy; they are not self-starting and require another scheme to retrieve solutions at initial time steps.
	
We propose a generalized-$\alpha$ method with an arbitrary order of approximation that provides A-stability parabolic time marching. Our scheme completes a Taylor expansion adding higher-order terms to the residual to obtain an auxiliary system to solve. We first analyze the amplification matrix's spectral properties to establish the parameter values that result in unconditional stability and control the high-frequency numerical dissipation. Next, we allow for complex entries in the amplification matrix to prove the method's A-stability for arbitrary accuracy. Numerical experiments, verify that the stability regions remain unchanged for any accuracy order and show that the technique improves the behavior of the generalized-$\alpha$ method in the moderate-frequency regions. Lastly,~\citet{ hughes2012finite} examines a method's accuracy using the Cayley-Hamilton theory; a commonly used technique in the literature (cf.,~\cite{ jansen2000generalized, behnoudfar2019higher, chung1993time}) but limited to methods that result in $2\times 2$ and $3\times 3$ amplification matrices. We extend this methodology to encompass the general $k\times k$ amplification matrices that result from our time-marching scheme.
	
We organize the remainder of the paper as follows. Section~\ref{sec:PS} describes the problem under consideration. Section~\ref{sec:ho} details a third-order generalized-$\alpha$ method, proves its third-order accuracy in time and its unconditional stability. Section~\ref{sec:HO} introduces our $(3/2k)^{th}$ and $(3/2k+1/2)^{th}$-order accuracy methods, demonstrates  their unconditional stability region, and defines their numerical dissipation control parameters. Section~\ref{sec:Conc} summarizes our contributions.

\section{Problem Statement} \label{sec:PS}

We consider a parabolic, linear initial boundary-value problem: 
\begin{equation} \label{eq:pde}
  \begin{aligned}
    \frac{\partial u(x,t)}{ \partial t} - \nabla \cdot( \kappa\nabla u(x, t)) & = f(x,t), \qquad &&(x, t) \in \Omega \times (0, T], \\
    u(x, t) & = u_D, \qquad \qquad &&x \in \partial \Omega, \\
    u(x, 0) & = u_0, \qquad \quad &&x \in \Omega.
  \end{aligned}
\end{equation}
Let $\Omega =(0, 1)^d \subset \mathbb{R}^d, d=1,2,3,$ be an open bounded domain with Lipschitz  boundary $\partial \Omega$. $\nabla $ is the spatial gradient operator, and $\kappa\in L^\infty(\Omega)$ is the diffusivity coefficient. The source function $f$, the initial data $u_0$, and the Dirichlet boundary condition $u_D$ are given and assumed regular enough so that the problem admits a weak solution. 
	
\subsection{Spatial discretization}

We use a spatial finite element discretization. We define $\mathcal{P}_h$ as a partition of the domain $\Omega$ into elements $K$ and obtain $\Omega_h := \bigcup_{K \in \mathcal{P}_h} K$. Following standard notation for the Lebesgue and Sobolev spaces, we assume $V_h^p$ as a finite-dimensional space composed of polynomial functions with order $p\geq1$ defined on $\Omega_h$. Then multiply the hyperbolic equation~\eqref{eq:pde} with a sufficiently regular test function $w_h \in V^p_h$ (here, $w_h \in H_0^1(\Omega)$), integrate over $\Omega_h$, and apply the divergence theorem to obtain the semi-discretized form of the problem as: 
\begin{equation} \label{eq:wf}
  a(w_h, \dot u_h) + b(w_h, u_h) = \ell(w_h), \qquad w_h \in V_h^p(\Omega), \ t>0,
\end{equation}
where $\dot u = \frac{\partial u}{\partial t}$ and 
\begin{equation}
  a(w, v) = (w, v)_\Omega, \quad  b(w, v) =  (\kappa \nabla w, \nabla u)_\Omega, \quad \ell(w) = (w, f)_\Omega,  \qquad \forall w,v \in V_h^p(\Omega).
\end{equation}
	
We approximate $u(x,t)$ for each fixed $t$ by a function $u_h(x, t)$ that belongs to a finite-dimensional space $V_h^p$. The spatial discretization leads to the following variational formulation:
\begin{equation} \label{eq:disc}
  \begin{aligned}
    \begin{cases}
      \text{Find } u_h(t) = u_h(\cdot, t) \in V_h^p(\Omega) \qquad &\text{ for }t>0: \\
      a(w_h, \dot u_h) + b(w_h, u_h) = \ell(w_h), \qquad &\forall w_h \in V_h^p
    \end{cases}
  \end{aligned}
\end{equation}
with $u_h(0)$ being the interpolation of $u_0$ in $V^h$. The matrix form of the discrete problem~\eqref{eq:disc} becomes:
\begin{equation} \label{eq:mp}
  M \dot{U} + KU = F,
\end{equation}
where $M$ and $K$ are the mass and stiffness matrices, $U$ is the vector of the unknowns, and $F$ is the source vector. The initial condition is 
\begin{equation} \label{eq:u0}
  U(0) = U_0,
\end{equation}
where $U_0$ is the given vector of initial condition $u_{0,h}$.

\begin{remark}
  Herein, we propose a high-order generalized-$\alpha$ by introducing a general spatial discretization, leading to the matrix problem~\eqref{eq:mp}. Nevertheless, one can also apply our construction to any time-dependent semi-discretized problem.
\end{remark}
\begin{remark}
  In problem~\eqref{eq:disc}, for simplicity, we  consider constant $\kappa$ and assume that the solution $u(\cdot\,,t)$ satisfies Dirichlet boundary conditions. Whereas, one requires slight modifications of the discrete bilinear and linear functions for heterogeneous diffusivity and non-homogeneous boundary conditions (see, for example,~\cite{ hughes2012finite}).
\end{remark}

\subsection{Time-discretization using generalized-$\alpha$ method}

Consider a uniform partitioning of the time interval $[0,T]$ with a grid size $\tau$: $0 = t_0 < t_1 < \cdots < t_N = T$ and denote by $U_n, V_n$  the approximations to $U(t_n), \dot U(t_n)$, respectively. The generalized-$\alpha$ method for~\eqref{eq:mp} at time-step $n$ is ~\cite{ jansen2000generalized}: 
\begin{equation} \label{eq:galpha}
  \begin{aligned}
    M V_{n+\alpha_m} + K U_{n+\alpha_f} & = F_{n+\alpha_f}, \\
    U_{n+1} & = U_n + \tau V_n + \tau \gamma \llbracket V_n \rrbracket, \\
    V_0 & = M^{-1} (F_0 - K U_0),
  \end{aligned}
\end{equation}
where
\begin{equation} \label{eq:mf}
  \begin{aligned}
    F_{n+\alpha_f} & = F(t_{n+\alpha_f}), \\
    W_{n+\alpha_g} & = W_n + \alpha_g \llbracket W_n \rrbracket,\qquad \llbracket W_n \rrbracket & = W_{n+1} - W_n, \quad W = U, V, \quad g=m, f.
  \end{aligned}
\end{equation}
According to~\eqref{eq:galpha}, the method requires a two-step computation; the first one solves an implicit system to find $\llbracket V_n \rrbracket$, the second one uses the second equation in~\eqref{eq:galpha} to update $U_{n+1}$ explicitly. We guarantee the scheme's second-order accuracy in time by setting $\gamma = \frac{1}{2} + \alpha_m - \alpha_f$; we also control the numerical dissipation using the following parameter definition~\cite{ jansen2000generalized}:
\begin{equation} \label{eq:hfd}
  \alpha_m = \frac{1}{2} \Big( \frac{3 - \rho_\infty}{1+\rho_\infty} \Big), \qquad \alpha_f = \frac{1}{1+\rho_\infty}.
\end{equation}
where, $\rho_\infty\in[0,1]$ is a user-control parameter. 

\section{Third-order generalized-$\alpha$ method} \label{sec:ho}

The sub-step time-marching affects the accuracy of the generalized-$\alpha$ method~\eqref{eq:galpha}; we use a Taylor expansion to analyze the truncation error and obtain $\mathcal{O}(\tau^3)$. We extend this analysis to derive higher-order representations that rely on Taylor expansions to solve this problem; see also the discussions in~\cite{ deng2019high, behnoudfar2019higher}. For this purpose, we introduce higher-order terms and define $\mathcal{L}^a(w)$ as the $a$-th order derivative of the function $w$ in time. For example, we derive a third-order generalized-$\alpha$ method for the semi-discrete equation~\eqref{eq:mp}, where $A_n$ approximates $\frac{\partial^2 u}{\partial t^2}$. Using this substitution, we readily obtain $M\mathcal{L}^{1}(A_n)^{\alpha_2}+K A_n^{\alpha_f}=\mathcal{L}^{2}(F_n)^{\alpha_f} $ by taking two temporal derivatives from the first equation of~\eqref{eq:galpha}. Assuming sufficient smoothness of the solution and forcing on the time interval under analysis, we propose a method that solves
\begin{equation} \label{eq:4a}
  \begin{aligned}
    MV_{n}^{\alpha_1}&=-K U_{n+1}+F_{n+1}, \\
    M\mathcal{L}^{1}(A_n)^{\alpha_2}&=-K A_n^{\alpha_f}+\mathcal{L}^{2}(F_n)^{\alpha_f} ,
  \end{aligned}
\end{equation}
with updating conditions
\begin{equation} \label{eq:4aup}
  \begin{aligned}
    U_{n+1}  &= U_n + \tau V_n + \frac{\tau^2}{2} A_n + \frac{\tau^3}{6} \mathcal{L}^{1}(A_n)+ \gamma_1 \tau Q_{n}, \\
    A_{n+1}&=A_{n} + \tau \mathcal{L}^{1}(A_n) + \tau \gamma_2  \llbracket {\mathcal{L}^{1}(A_{n})} \rrbracket, 
  \end{aligned}
\end{equation}
where
\begin{equation} 
  \begin{aligned}
    Q_{n}&=V_{n+1}-V_n-\tau A_n-\frac{\tau^2}{2} \mathcal{L}^{1}(A_n),\\
    V_{n}^{\alpha_1}&=V_n+\tau A_{n}+\frac{\tau^2}{2} \mathcal{L}^{1}(A_{n})+\alpha_1 Q_{n},\\[0.2cm]
    \mathcal{L}^{1}(A_n)^{\alpha_2}&=\mathcal{L}^{1}(A_n)+\alpha_2\llbracket {\mathcal{L}^{1}(A_{n})} \rrbracket,\\
    A_n^{\alpha_f}&=A_n+\alpha_f \llbracket {A_{n}} \rrbracket.
  \end{aligned}
\end{equation}
The initial data are also obtained by using the given information $U_0$ as:
\begin{equation}
  \begin{aligned}
    V_0 & = M^{-1} (F_0 - K U_0),\\
    A_0 & = M^{-1} (\mathcal{L}^1(F_0) - K V_0),\\
    \mathcal{L}^1(A_0) & = M^{-1} (\mathcal{L}^2(F_0) - K A_0).
  \end{aligned}
\end{equation}
Next, we define the parameters to guarantee the stability and the third-order of accuracy. 

\subsection{Order of accuracy in time}

Now, we determine the conditions on parameters $\gamma_1$ and $\gamma_2$ such that equations~\eqref{eq:4a}-\eqref{eq:4aup} deliver third-order accuracy in time, which renders the following result.
\begin{theorem} \label{thm:3o}
  Assuming the solution is sufficiently smooth with respect to time, the method in~\eqref{eq:4a} with the update~\eqref{eq:4aup} is third-order accurate in time given
  \begin{equation} \label{eq:3ov1}
    \gamma_1=\alpha_1-\frac{1}{2}, \qquad \qquad 	\gamma_2=\frac{1}{2}-\alpha_{f}+\alpha_2.
  \end{equation}
\end{theorem}
\begin{proof}
  Substituting~\eqref{eq:4aup} into~\eqref{eq:4a}, we obtain an equation system for each time step:
  \begin{equation}\label{eq:proof}
    A\bold{ U_{n+1}}=B \bold{U_n}+\bold{F_{n+\alpha_f}},
  \end{equation}
  For simplicity, we assume the matrix problem~\eqref{eq:proof} has one spatial degree of freedom. Then, letting $\bold{U}_{n}^T=
      \begin{bmatrix}
        U_{n},\
        \tau V_{n},\
        \tau^2 A_{n},\
        \tau^3\mathcal{L}^{1}(A_{n})
      \end{bmatrix}^T,$ we get: 

  \begin{equation}
    \begin{aligned}A&=
      \begin{bmatrix}
        1 & -\gamma_1&0&0 \\
        \tau \lambda & \alpha_1&0&0 \\
        0 & 0&1&-\gamma_2 \\
        0 & 0&\tau\alpha_f \lambda&\alpha_2
      \end{bmatrix},\
      B&=
      \begin{bmatrix}
        1 & 1-\gamma_1& \frac{1}{2}-\gamma_1& \frac{1}{6}-\frac{1}{2}\gamma_1\\
        0 & \alpha_1-1& \alpha_1-1&\frac{1}{2}( \alpha_1-1) \\
        0 & 0&1&1-\gamma_2 \\
        0 & 0&-\tau(1-\alpha_{f}) \lambda&\alpha_2-1 
      \end{bmatrix}.
    \end{aligned}
  \end{equation}
  Without loss of generality, we set $\bold{F_{n+\alpha_f}=0}$. Thus, the amplification matrix $G$ reads:
  \begin{equation} \label{eq:ampm}
    G=A^{-1}B.
  \end{equation}
  This matrix-matrix multiplication results in an upper-block diagonal matrix:
  \begin{equation}
    G=\begin{bmatrix}
      G_1&H\\
      0&G_2
    \end{bmatrix},
  \end{equation}
  with 
  \begin{align}
    G_1&=\theta_1\begin{bmatrix}
      {\alpha_1} & {\alpha_1-\gamma_1}\\
      -{\theta} & {\alpha_1+ (\gamma_1-1) \theta-1}
    \end{bmatrix},\\\label{eq:G2}
    G_2&=\theta_2\begin{bmatrix}
      {\alpha_2+(\alpha_f-1) \gamma_2 \theta}& {\alpha_2-\gamma_2} \\
      -{\theta}& {\alpha_2+\alpha_f (\gamma_2-1) \theta-1}
    \end{bmatrix},
  \end{align}
  where $\theta_1=(\alpha_1+ \gamma_1 \theta)^{-1}$ and $\theta_2=(\alpha_2+\alpha_f \gamma_2 \theta)^{-1}$. Additionally, we define $\theta:=\tau \lambda_\theta $ with $\lambda_\theta$ being the eigenvalues of $M^{-1}K$ related to the spatial discretization. Then, the $4\times 4$ matrix G becomes:
  \begin{equation}\label{eq:C4}
    \begin{aligned}
      \textbf{0}&=U^{n+1}-(\tr G)U^n+\frac{1}{2}\left((\tr
        G)^2-\tr(G^2)\right)U^{n-1}\\
      &-\frac{1}{6}\left((\tr G)^3-3\tr(G^2)(\tr G)+2\tr (G^3)\right)U^{n-2}+\det (G) U^{n-3},
    \end{aligned}
  \end{equation}
  where $(\tr G)$ is the trace of the matrix $G$. The identity~\eqref{eq:C4} is a consequence of Cayley-Hamilton Theorem for a $4 \times 4$ matrix. We detail these computations when we analyze a general $k \times k$ matrix in the next section.  Later, we substitute the Taylor expansions of $U^{n+1},\,U^{n-1},\,U^{n-2}$ in time with the truncation error of $\mathcal{O}(\tau^4)$ into~\eqref{eq:C4} . Then, one can verify that these parameter definitions~\eqref{eq:3ov1} cancels the low-order terms and delivers third-order accuracy.
\end{proof}

\begin{remark}
  The original generalized-$\alpha$ method delivers second-order accuracy and dissipation control using two equations; similarly, we obtain a third-order rather than fourth-order to maintain the unconditional stability and dissipation control. For example, we get fourth-order accuracy by setting $\gamma_2=0$, which results in an explicit-implicit method with a CFL condition. Therefore, we omit the details for other possible choices of $\gamma_i$ and only consider~\eqref{eq:3ov1} that leads to unconditional A-stability.  
\end{remark}
\begin{remark}[Alternative proof for accuracy order]
  Accordingly, due to the structure of $G$, we only need to study diagonal blocks' behaviour in our analysis. Therefore, we require that the high-order unknowns $A_n, \mathcal{L}^{1}(A_{n})$, associated with the lower diagonal-block $G_2$ be second-order accurate~\cite{ deng2019high}:
  \begin{equation} \label{eq:a40}
    A_{n+1} - \tr (G_2)\, A_n - \det(G_2)\,A_{n-1} = 0.
  \end{equation}
  Then, further substitution of Taylor expansions of $A_{n+1}$ and $A_{n-1}$ in time as: 
  \begin{equation} \label{eq:te3}
    \begin{aligned}
      A_{n+1} & = A_{n}+\tau \mathcal{L}^{1}(A_{n}) +\dfrac{\tau^2 }{2}\mathcal{L}^{2}(A_{n})+ \mathcal{O}(\tau^3), \\
      A_{n-1} & =A_{n}-\tau \mathcal{L}^{1}(A_{n})+\dfrac{\tau^2 }{2}\mathcal{L}^{2}(A_{n}) +\mathcal{O}(\tau^3),
    \end{aligned}
  \end{equation}
  requires $\gamma_2=\frac{1}{2}-\alpha_{f}+\alpha_2$ to guarantee the second-order accuracy in time of $A_{n+1}$. Then, similarly to~\eqref{eq:a40} and using the upper diagonal-block $G_1$, we obtain: 
  \begin{equation} \label{eq:a4new}
    U_{n+1} - \tr(G_1)\, U_n + \det(G_1)\,U_{n-1} = 0.
  \end{equation}
  We obtain the third-order accuracy in time using Taylor expansions with truncation error of $\mathcal{O}(\tau^4)$ for $U_{n+1}$ and $U_{n-1}$, which allows to get
  \begin{equation} \label{eq:te6}
    \begin{aligned}
      U_{n+1} & = U_n + \tau V_n + \underbrace{\dfrac{\tau^2}{2} A_n+\dfrac{\tau^3}{6}\mathcal{L}^{1}(A_{n})}_\mathcal{R} +\mathcal{O}(\tau^4), \\
      U_{n-1} & = U_n - \tau V_n+\underbrace{\dfrac{\tau^2}{2} A_n-\dfrac{\tau^3}{6}\mathcal{L}^{1}(A_{n})}_\mathcal{R} +\mathcal{O}(\tau^4) ,
    \end{aligned}
  \end{equation}
  Here, the $G_2$ equations show that $\mathcal{R}$  has second-order of accuracy and define a residual term. Thus, by following a similar approach, we verify that the remaining terms have second-order accuracy in time by setting $ \gamma_1=\alpha_1-\frac{1}{2}$. Then, we add the residuals to the second-order accurate solution, to have the truncation error of $\mathcal{O}(\tau^4)$ and consequently, a third-order accurate scheme in time, which completes the proof.
\end{remark}

\subsection{Stability analysis and eigenvalue control}\label{sec:anal}

To obtain an unconditionally stable method, we bound the spectral radius of the amplification $G$ by one; thus, we first determine the eigenvalues of $G$ as: 
\begin{equation}\label{eq:det3}
  \begin{aligned}
    0&=\det \left(G-\tilde{\lambda}I\right)=\det\begin{bmatrix}
      G_1-\tilde{\lambda}I&H\\
      0&G_2-\tilde{\lambda}I
    \end{bmatrix}\\
    &=\det\left(G_1-\tilde{\lambda}I\right)\cdot \det\left(G_2-\tilde{\lambda}I\right),
  \end{aligned}
\end{equation}
with $I$ and $\tilde{\lambda}$ denoting the corresponding identity matrix and eigenvalues of the matrix, respectively. Thus, we obtain a solution for~\eqref{eq:det3} by solving two uncoupled problems $\det\left(G_i-\tilde{\lambda}I\right)=0,\, i=1,\,2$. For this, we have~\cite{ horn1990matrix}:
\begin{equation}\label{eq:stab}
  \det\left(G_i-\tilde{\lambda}I\right)=\det\left(\Lambda^i_{11}-\tilde{\lambda}I\right)\cdot \det\left(\Lambda^i_{22}-\tilde{\lambda}I-\Lambda^i_{21}\left(\Lambda^i_{11}-\tilde{\lambda}I\right)^{-1}\Lambda^i_{12}\right),
\end{equation}
where $\Lambda^i_{jk}$ is the $jk$ component of $G_i$. Therefore, substituting~\eqref{eq:G2} into~\eqref{eq:stab}, we consider $\det\left(\Lambda^2_{11}-\tilde{\lambda}I\right)=0$ which leads to the following bound on $\tilde{\lambda}$:
\begin{equation}
  -1\leq \left(\alpha_2+(\alpha_f-1) \gamma_2 \theta\right){\theta_2}\leq 1\implies 0\leq \gamma_2 \theta {\theta_2}\leq 2.
\end{equation}
The left inequality is already satisfied since all parameters are non-negative and the matrices $M$ and $K$ obtained after the spatial discretizations are positive definite. We rewrite the right-hand side of the inequality as:
\begin{equation}\label{eq:ineq}
  \gamma_2\theta (1-2\alpha_{f})\leq 2\alpha_2.
\end{equation} 
To satisfy~\eqref{eq:ineq} unconditionally, one requires $\dfrac{1}{2}\leq\alpha_{f}$. Next, we solve  $$\det\left(\Lambda^2_{22}-\tilde{\lambda}I-\Lambda^2_{21}\left(\Lambda^2_{11}-\tilde{\lambda}I\right)^{-1}\Lambda^2_{12}\right)=0$$ in~\eqref{eq:stab} using:
\begin{equation}
  \left(1-{\gamma_2\theta}\theta_2\tilde{\lambda}\right)\cdot \left(1-(1+\alpha_f\gamma_2\theta)\theta_2\tilde{\lambda}\right)+\left(1-{\theta}\theta_2\right)\cdot  \left(1-(\gamma_2+\alpha_f\gamma_2\theta)\theta_2\right)=0.
\end{equation} 
We omit the details, and to bound the spectral radius by one, it is sufficient to impose: 
\begin{equation}
  \alpha_2\geq \alpha_{f}\geq\dfrac{1}{2}. 
\end{equation}
To solve~\eqref{eq:stab} for $i=1$, one can follow the same steps. Thus, we solve  $$\det\left(\Lambda^1_{11}-\tilde{\lambda}I\right)=0$$ and bound the spectral radius as:
\begin{equation}
  -1\leq \alpha_1 {\theta_1}\leq 1 \implies \gamma_1\theta\leq 2\left(\alpha_1+\gamma_1\theta\right).
\end{equation}
Furthermore,  substituting $$\det\left(\Lambda^1_{22}-\tilde{\lambda}I-\Lambda^1_{21}\left(\Lambda^1_{11}-\tilde{\lambda}I\right)^{-1}\Lambda^1_{12}\right)=0$$ in~\eqref{eq:stab} implies that $\alpha_1\geq 1$. We omit the details for brevity. We calculate the eigenvalues for the case $\theta \to \infty$ equal to $\rho^\infty_1$ to control the eigenvalues in high frequency regions. Accordingly, the eigenvalues of the amiplification matrix read:
\begin{equation}
  \lambda_1=0,\quad \lambda_2=\frac{\gamma_1-1}{\gamma_1},\quad\lambda_3=\frac{\alpha_f-1}{\alpha_f},\quad\lambda_4=\frac{\gamma_2-1}{\gamma_2}.\quad
\end{equation}
To provide control on the numerical dissipation, following closely the analysis in~\cite{ chung1993time,jansen2000generalized,behnoudfar2018variationally}, we set $\lambda_2=\rho^\infty_1$ and $\lambda_3=\lambda_4=\rho^\infty_2$ and find corresponding expressions for $\alpha_1,\,\alpha_2,\,\alpha_f$ as:
\begin{equation}
  \begin{aligned}
    \alpha_1&=\frac{1}{2}\left(\frac{3+\rho^\infty_1}{1+\rho^\infty_1}\right),\\
    \alpha_2&=\frac{1}{2}\left(\frac{3-\rho^\infty_2}{1+\rho^\infty_2}\right),\\
    \alpha_f&=\frac{1}{1+\rho^\infty_2}.
  \end{aligned}
\end{equation}
Therefore, setting $0\leq \rho^\infty_1,\rho^\infty_2\leq1 $, one controls the eigenvalues of the amplification matrix and the high-frequency damping. Figure~\ref{fig:dis} shows the behaviour of these eigenvalues; for large $\theta$, the eigenvalues $\lambda_{1,2}$ of the first block of the amplification matrix approach $0$ and $\rho^\infty_1$ and eigenvalues of the second  block $\lambda_{3,4}$ converge to $\rho^\infty_2$.
\begin{figure}[!ht]	
  \subfigure[$\rho^\infty_1=0.8$, $ \rho^\infty_2=0.2$ ]{\centering\includegraphics[width=6.5cm]{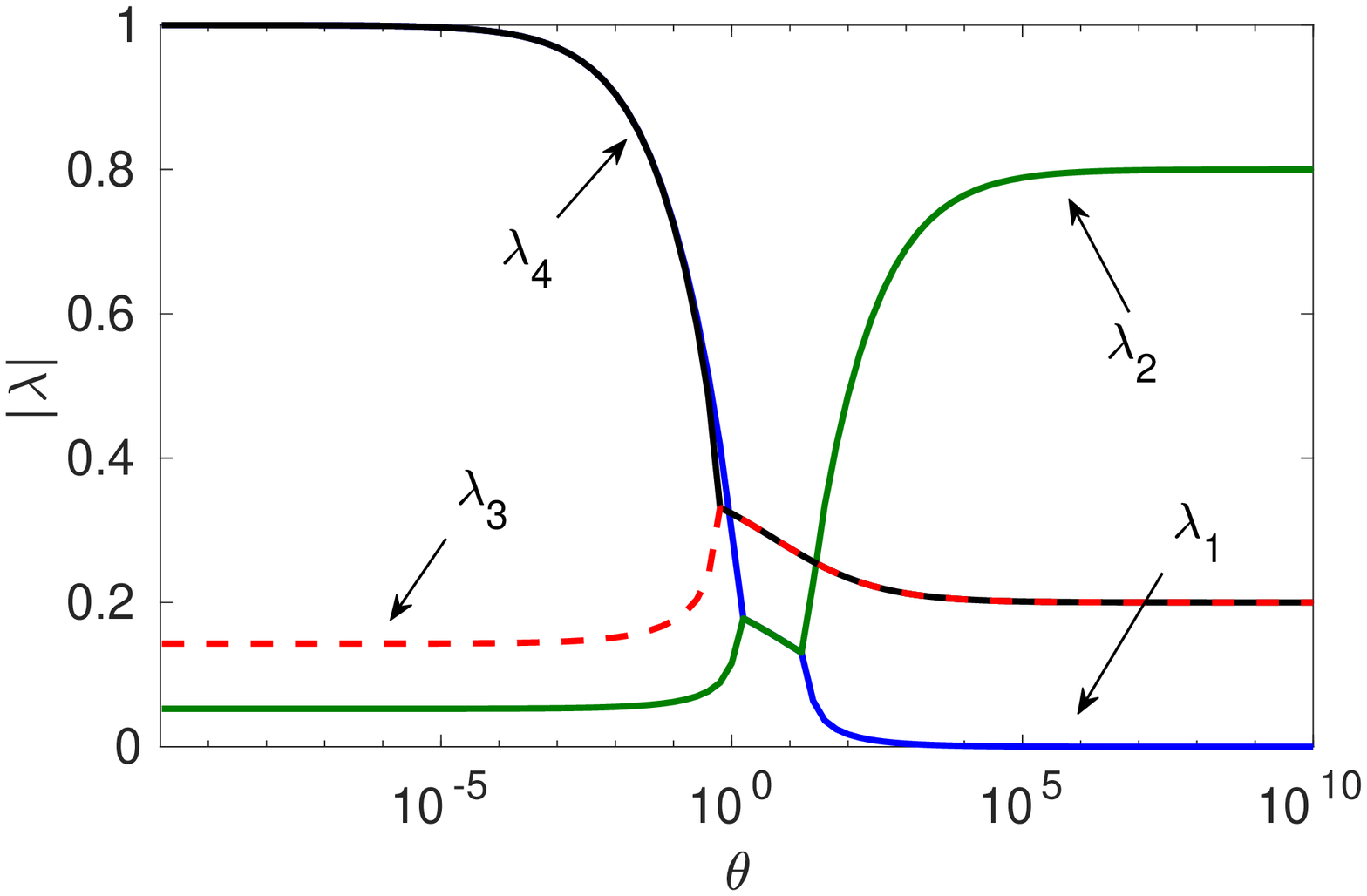} }
  \subfigure[$\rho^\infty_1=\rho^\infty_2=0$ ]{\centering\includegraphics[width=6.42cm]{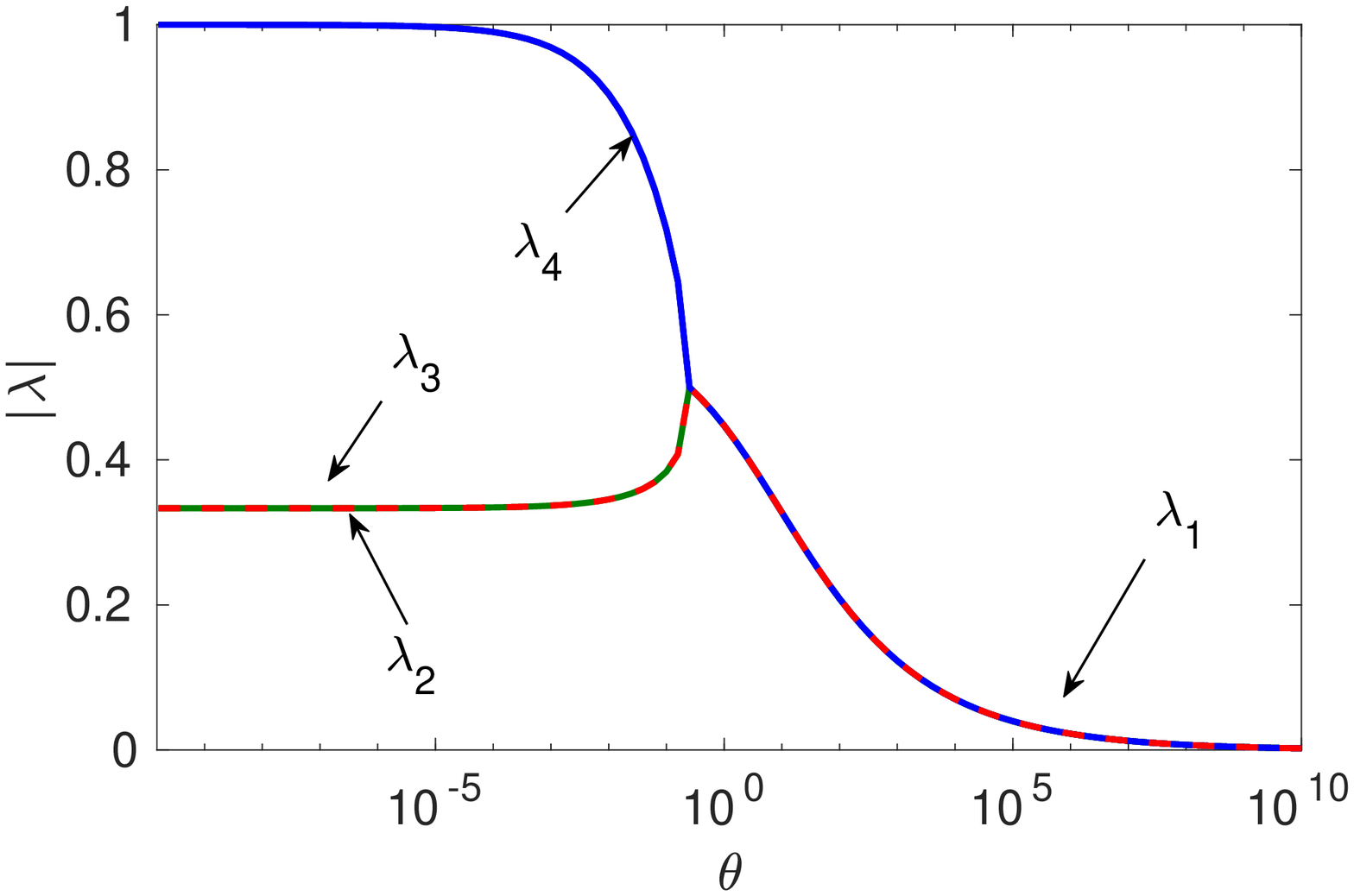} }
  \caption{{The eigenvalues of the amplification matrix~\eqref{eq:ampm}.}}
  \label{fig:dis}
\end{figure}
	
\section{Higher-order accuracy in time}\label{sec:HO}

In this section, we propose our method delivering higher-order of accuracy in time. In general, to solve the semi-discretized matrix problem~\eqref{eq:mp}, we formulate a high-order generalized-$\alpha$ method by solving $k\geq2$ equations as:
\begin{equation} \label{eq:ho}
  \begin{aligned}
    MV_{n}^{\alpha_1}&=-K U_{n+1}+F_{n+1}, \\
    M\mathcal{L}^{2j-3}(A_n)^{\alpha_j}&=-K \mathcal{L}^{2j-4}(A_{n+1})+\mathcal{L}^{2j-2}(F_{n+1}), \qquad j=2, \cdots, k-1, \\
    M\mathcal{L}^{2k-3}(A_n)^{\alpha_k}&=-K \mathcal{L}^{2k-4}(A_{n})^{\alpha_f}+\mathcal{L}^{2k-2}(F_n)^{\alpha_f}, 
  \end{aligned}
\end{equation}
to obtain $(\frac{3}{2}k)^{th}$ order of accuracy for even $k$ and $(\frac{3}{2}k+\frac{1}{2})^{th}$ for odd $k$. Next, we update the system explicitly using the following:
\begin{equation} \label{eq:ho2}
  \begin{aligned}
    U_{n+1} & = U_n + \tau V_n + \frac{\tau^2}{2} A_n + \frac{\tau^3}{6} \mathcal{L}^{1}(A_n)+ \cdots+\frac{\tau^{2k-1}}{(2k-1) !} \mathcal{L}^{2k-3}(A_n)+\tau \gamma_1 Q_{n,1}, \\
    \mathcal{L}^{2j-4}(A_{n+1})&= \mathcal{L}^{2j-4}(A_{n}) + \cdots+\frac{\tau^{2k-2j+1}}{(2k-2j+1)!} \mathcal{L}^{2k-3}(A_n)+\tau \gamma_j Q_{n,j},\quad j=2, \cdots, k-1, \\[0.2cm] 
    \mathcal{L}^{2k-4}(A_{n+1}) &=\mathcal{L}^{2k-4}(A_{n}) + \tau \mathcal{L}^{2k-3}(A_{n}) + \tau \gamma_k \cdot \del{\mathcal{L}^{2k-3}(A_{n})},
  \end{aligned}
\end{equation}
where we have
\begin{equation} \label{eq:ho3}
  \begin{aligned}
    Q_{n,1}&=V_{n+1}-V_n-\tau A_{n}-\cdots-\frac{\tau^{2k-2}}{(2k-2) !} \mathcal{L}^{2k-3}(A_{n}),\\
    V_{n}^{\alpha_1}&=V_n+\tau A_{n}+\cdots+\frac{\tau^{2k-1}}{(2k-1) !} \mathcal{L}^{2k-3}(A_{n})+\alpha_1 Q_{n,1},\\
    Q_{n,j}&=\mathcal{L}^{2j-4}(A_{n+1})-\mathcal{L}^{2j-4}(A_n)-\cdots-\frac{\tau^{2k-2j+1}}{(2k-2j+1)!} \mathcal{L}^{2k-3}(A_n),
  \end{aligned}
\end{equation}
and for $j=2, \cdots, k-1$
\begin{equation} \label{eq:ho3.1}
  \begin{aligned}
    \mathcal{L}^{2j-3}(A_{n})^{\alpha_j}&=\mathcal{L}^{2j-3}(A_n)+\cdots+\frac{\tau^{2k-2j+1}}{(2k-2j+1)!} \mathcal{L}^{2k-3}(A_n) +\alpha_jQ_{n,j},\\
    \mathcal{L}^{2k-3}(A_n^{\alpha_k})&=\mathcal{L}^{2k-3}(A_n)+\alpha_k \cdot  \del{ \mathcal{L}^{2k-3}(A_{n})},\\
    \mathcal{L}^{2k-4}(A_n^{\alpha_f})&=\mathcal{L}^{2k-4}(A_n)+\alpha_f \cdot  \del{\mathcal{L}^{2k-4}(A_{n})}.
  \end{aligned}
\end{equation}
For $k=1, 2$, this reduces to the second- and third-order generalized-$\alpha$ methods, respectively. Next, we define the parameters associated with the accuracy of the method.
	
\subsection{Analysing the order of accuracy}
We analyse the accuracy of the method that~\eqref{eq:ho}-\eqref{eq:ho3.1} define by first deriving the amplification matrix. Similarly to the third-order method, we substitute~\eqref{eq:ho2} into~\eqref{eq:ho} and find a matrix system:
\begin{equation}\label{eq:eigHo}
  L \bfs{U}_{n+1} = R \bfs{U}_n+\bold{F_{n+\alpha_f}}.
\end{equation}
Therefore, the amplification matrix corresponding to a $k$-equation system becomes $G=L^{-1}R$ with:
\begin{equation}\label{eq:HOampblock}
  G=\begin{bmatrix}
    G_1 &\Xi_{12} &\cdots&\cdots& \Xi_{1k}\\
    \boldsymbol{0}&G_2&\Xi_{23} &\cdots& \Xi_{2k}\\
    \vdots& &\ddots& \\
    \boldsymbol{0}&\boldsymbol{0}&\cdots & G_{k-1}& \Xi_{k-1}\\
    \boldsymbol{0}&\boldsymbol{0}&\cdots & \boldsymbol{0}& G_k
  \end{bmatrix},
\end{equation}
where
\begin{align}
  G_j&=\theta_j\begin{bmatrix}
    {\alpha_j} & {\alpha_j-\gamma_j}\\
    -{\theta} & {\alpha_j+ (\gamma_j-1) \theta-1}
  \end{bmatrix},\qquad j=1, \cdots, k-1,\\\label{eq:GHO2}
  G_k&=\theta_k\begin{bmatrix}
    {\alpha_k+(\alpha_f-1) \gamma_k \theta}& {\alpha_k-\gamma_k} \\
    -{\theta}& {\alpha_k+\alpha_f (\gamma_k-1) \theta-1}
  \end{bmatrix},
\end{align}
and denoting $\theta_j=(\alpha_j+ \gamma_j \theta)^{-1}$, $\theta_k=(\alpha_k+\alpha_f \gamma_k \theta)^{-1}$.

\begin{theorem}
  The method that equations~\eqref{eq:ho}-\eqref{eq:ho3.1} define for semi-discterized system~\eqref{eq:disc}, delivers $(\frac{3}{2}k)^{th}$ order of accuracy for even $k$ and $(\frac{3}{2}k+\frac{1}{2})^{th}$ order for odd $k$ in time by setting:
  \begin{equation} \label{eq:3ovHo}
    \begin{aligned}
      \gamma_j&=\alpha_j-\frac{1}{2}, \qquad \text{for }\, j=1,\cdots,k-1,\\
      \gamma_k&=\frac{1}{2}-\alpha_{f}+\alpha_k.
    \end{aligned}
  \end{equation}

  \begin{proof}
    For a $k$-system of equations, we expand using Taylor series for $k$ unknowns $U^{n+1},\, U^{n-1},\, \cdots\,,U^{n-k+1}$ around $U^n$ in time with truncation error of $\mathcal{O}\left(\tau^{\frac{3}{2}k+1}\right)$ for even $k$, and $\mathcal{O}\left(\tau^{\frac{3}{2}k+\frac{3}{2}}\right)$ for odd $k$ as:
    \begin{equation}\label{eq:taylor}
      \begin{aligned}
        U^{n+1}&=U^n+\tau V^n+\frac{\tau^2}{2}A^n+\cdots,\\
        U^{n-1}&=U^n-\tau V^n+\frac{\tau^2}{2}A^n+\cdots,\\
        &\vdots\\
        U^{n-k+1}&=U^n-\tau({n-k+1}) V^n+\frac{\left(\tau({n-k+1})\right)^2}{2}A^n+\cdots.
      \end{aligned}
    \end{equation}
    Then, for second-order accurate schemes, \citet{ hughes2012finite} uses the amplification matrix's invariants and the Taylor expansions to analyze the accuracy order. The approach applies the Cayley--Hamilton theorem to the resulting $2\times 2$ and $3\times 3$ amplification matrices for parabolic and hyperbolic problems, respectively, (see, e.g.,~\cite{ jansen2000generalized, chung1993time}). Herein, for schemes general $k\times k$ amplification matrices, we generalize the analysis to determine the parameters such that the method delivers a desired order of accuracy. Having this in mind, we discuss our general approach in the next section. 
  \end{proof}
\end{theorem}

\subsubsection{Analysing the accuracy of a general system}

Herein, we introduce a technique to study the accuracy of a time-marching method with arbitrary order. For this, we first discuss the general form of Cayley–Hamilton theorem. That is, for a general $k\times k$ matrix $G$, the characteristic polynomial of $G$, abusing notation, is $p(\lambda)=\det(\lambda I_k-G)$, with $I_k$ denoting the $k\times k$ identity matrix, for which we rewrite the characteristic polynomial $p(\lambda)$ as~\cite{ horn1990matrix}:
\begin{equation}\label{eq:pol}
  {p}(\lambda)=\lambda^k+c_{k-1}\lambda^{k-1}+\cdots+c_{1}\lambda^{1}+c_0.
\end{equation}
Then, instead of the scalar variable $\lambda$, one can obtain a similar polynomial to~\eqref{eq:pol} with the matrix $G$ as:
\begin{equation}\label{eq:polG}
  {p}(G)=G^k+c_{k-1}G^{k-1}+\cdots+c_{1}G^{1}+c_0 I_k,
\end{equation}
for which the Cayley--Hamilton theorem states that the polynomial~\eqref{eq:polG} equals to the zero matrix, $p(G)=\textbf{0}$~\cite{ householder2013theory}. We provide further details on the determination of the coefficients $c$ in~\ref{app:1}. Next, we multiply~\eqref{eq:polG} by $U^{n-k+1}$ to have:
\begin{equation}\label{eq:accuracy1}
  G^k U^{n-k+1}+c_{k-1}G^{k-1}\,U^{n-k+1}+\cdots+c_{1}G^{1}\,U^{n-k+1}+c_0\, U^{n-k+1}=\textbf{0},
\end{equation}
given that $U^n=G^nU^0$~\cite{ deng2019high}, we have that:
\begin{equation}\label{eq:accuracy2}
  U^{n+1}+c_{k-1}\,U^{n}+\cdots+c_{1}\,U^{n-k}+c_0\, U^{n-k+1}=\textbf{0}.
\end{equation}
then, substituting~\eqref{eq:taylor} into~\eqref{eq:accuracy2} and collecting the terms lead to:
\begin{equation}
  (1+c_{k-1}+\cdots +c0)U^n+(1-c_{k-2}+\cdots+(-1)^{k+1} c_0)\tau V^n+\cdots=\textbf{0}.
\end{equation}
Using the problem definition, we have $V^n=\lambda_\theta U^n$, $A^n=\lambda_\theta^2 U^n$, which is true for all higher-order terms defined using $U^n$. Finally, we set the terms $\gamma_j$ and $\gamma_k$ to cancel lower-order terms to obtain the optimal accuracy. 

\subsection{Stability analysis}\label{sec:stab}

In this section we follow closely our previous discussions in~\ref{sec:anal} to establish the unconditional stability of the method that equations~\eqref{eq:ho}-\eqref{eq:ho3.1} define; thus, we calculate the eigenvalues of the amplification matrix $G$ in~\eqref{eq:HOampblock} as:
\begin{equation}\label{eq:det}
  \begin{aligned}
    0&=\det \left(G-\tilde{\lambda}I\right)=\det
    \begin{bmatrix}
      G_1-\tilde{\lambda}I &\Xi_{12} &\cdots&\cdots& \Xi_{1k}\\
      \boldsymbol{0}&G_2-\tilde{\lambda}I&\Xi_{23} &\cdots& \Xi_{2k}\\
      \vdots& &\ddots& \\
      \boldsymbol{0}&\boldsymbol{0}&\cdots & G_{k-1}-\tilde{\lambda}I& \Xi_{k-1}\\
      \boldsymbol{0}&\boldsymbol{0}&\cdots & \boldsymbol{0}& G_k-\tilde{\lambda}I
    \end{bmatrix}, \\
    &=\det\left(G_1-\tilde{\lambda}I\right)\cdot \det\left(G_2-\tilde{\lambda}I\right) \cdots \det\left(G_k-\tilde{\lambda}I\right).
  \end{aligned}
\end{equation}
Therefore, we bound the spectral radius of each diagonal block to guarantee the overall stability; thus, for $\det\left(G_j-\tilde{\lambda}I\right)=0,\, j=1,\,\cdots,\,k$, expression~\eqref{eq:stab} is valid. Similarly, defining  $\Lambda^j_{lm}$ as the $lm$ component of $G_j$ allows us to bound the spectral radius of $\det\left(\Lambda^k_{11}-\tilde{\lambda}I\right)=0$ by:
\begin{equation}\label{eq:ineqho}
  \gamma_k\theta (1-2\alpha_{f})\leq 2\alpha_k.
\end{equation} 
Furthermore,  $\det\left(\Lambda^k_{22}-\tilde{\lambda}I-\Lambda^k_{21}\left(\Lambda^k_{11}-\tilde{\lambda}I\right)^{-1}\Lambda^k_{12}\right)=0$ results in:
\begin{equation}\label{eq:eqho}
  \left(1-{\gamma_k\theta}\theta_k\tilde{\lambda}\right)\cdot \left(1-(1+\alpha_f\gamma_k\theta)\theta_k\tilde{\lambda}\right)+\left(1-{\theta}\theta_k\right)\cdot  \left(1-(\gamma_k+\alpha_f\gamma_k\theta)\theta_k\right)=0.
\end{equation} 
Therefore, to satisfy~\eqref{eq:ineqho} and bound the spectral radius in~\eqref{eq:eqho}, we impose the following (for details, see the analysis in the previous section):
\begin{equation}\label{eq:cond}
  \alpha_k\geq \alpha_{f}\geq\dfrac{1}{2}. 
\end{equation}
For the other diagonal blocks, $ j=1,\,\cdots,\,k-1$, the spectral radius of $\det\left(\Lambda^j_{11}-\tilde{\lambda}I\right)=0$, is already bounded. Besides, the equations $\det\left(\Lambda^j_{22}-\tilde{\lambda}I-\Lambda^j_{21}\left(\Lambda^j_{11}-\tilde{\lambda}I\right)^{-1}\Lambda^j_{12}\right)=0$ enforce:
\begin{equation}
  \alpha_j\geq1,\qquad j=1,\cdots,k-1.
\end{equation}
	
To control the numerical dissipation, we let $\theta \to \infty$ and obtain the eigenvalues of the amplification matrix $\eqref{eq:HOampblock}$ as:
\begin{equation}\label{eq:eigen}
  \begin{aligned}
    \lambda_{2j-1}&=0,\qquad \quad  &&\lambda_{2j}=\frac{\gamma_j-1}{\gamma_j},\qquad j=1,\cdots,k-1,\\
    \quad\lambda_{2k-1}&=\frac{\alpha_f-1}{\alpha_f},\qquad \quad &&\lambda_{2k}=\frac{\gamma_k-1}{\gamma_k}.
  \end{aligned}
\end{equation}
Then, we set $\lambda_{2j}=\rho^\infty_j$ and $\lambda_{2k-1}=\lambda_{2k}=\rho^\infty_k$ and find corresponding expressions for $\alpha_j,\,\alpha_k,\,\alpha_f$ as:
\begin{equation}
  \begin{aligned}
    \alpha_j&=\frac{1}{2}\left(\frac{3+\rho^\infty_1}{1+\rho^\infty_1}\right),\qquad j=1,\cdots,k-1,\\
    \alpha_2&=\frac{1}{2}\left(\frac{3-\rho^\infty_2}{1+\rho^\infty_2}\right),\\
    \alpha_f&=\frac{1}{1+\rho^\infty_2}.
  \end{aligned}
\end{equation}
As for the third-order method, choosing $0\leq \rho^\infty_j,\rho^\infty_k\leq1 $, one controls the dissipation in the high-frequency range while minimizing the dissipation in the low-frequency ones. 
	
\begin{remark}
  Setting $\rho^\infty_1 =\rho^\infty_2 =\cdots =\rho^\infty_k =\rho^\infty$, allows us to have a one-parameter family of methods with high accuracy. Additionally, the spectral radius of the system approaches to $\rho^\infty$ in the high-frequency regions.   
\end{remark} 

Figure~\ref{fig:spec} presents numerical evidence; the method's spectral behaviour for $k\geq2$ is independent of the accuracy order. Furthermore, in comparison with the second-order generalized-$\alpha$ method, our generalization improves the spectral behaviour in the mid-frequency regions (e.g., compare the spectral radius for $\rho^\infty=0.5$). Our method prevents extra damping in these regions with moderate frequency in the second-order generalized-$\alpha$ by approaching the spectral radius to zero for $\rho^\infty=0.5$.

\begin{figure}[!ht]	
  \subfigure[Second-order method ]{\centering\includegraphics[width=7.15cm]{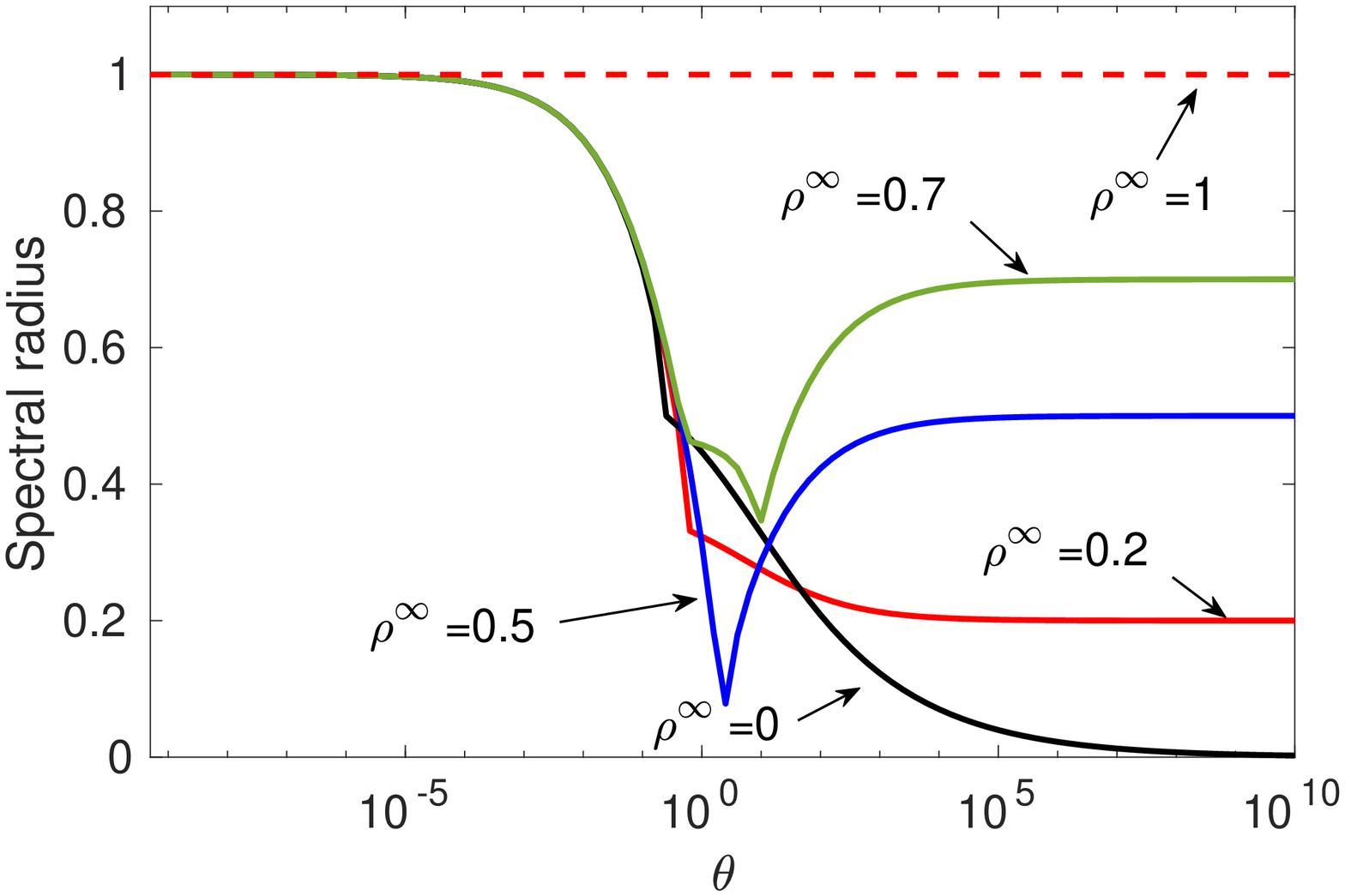} }
  \subfigure[Third-order method ]{\centering\includegraphics[width=7.15cm]{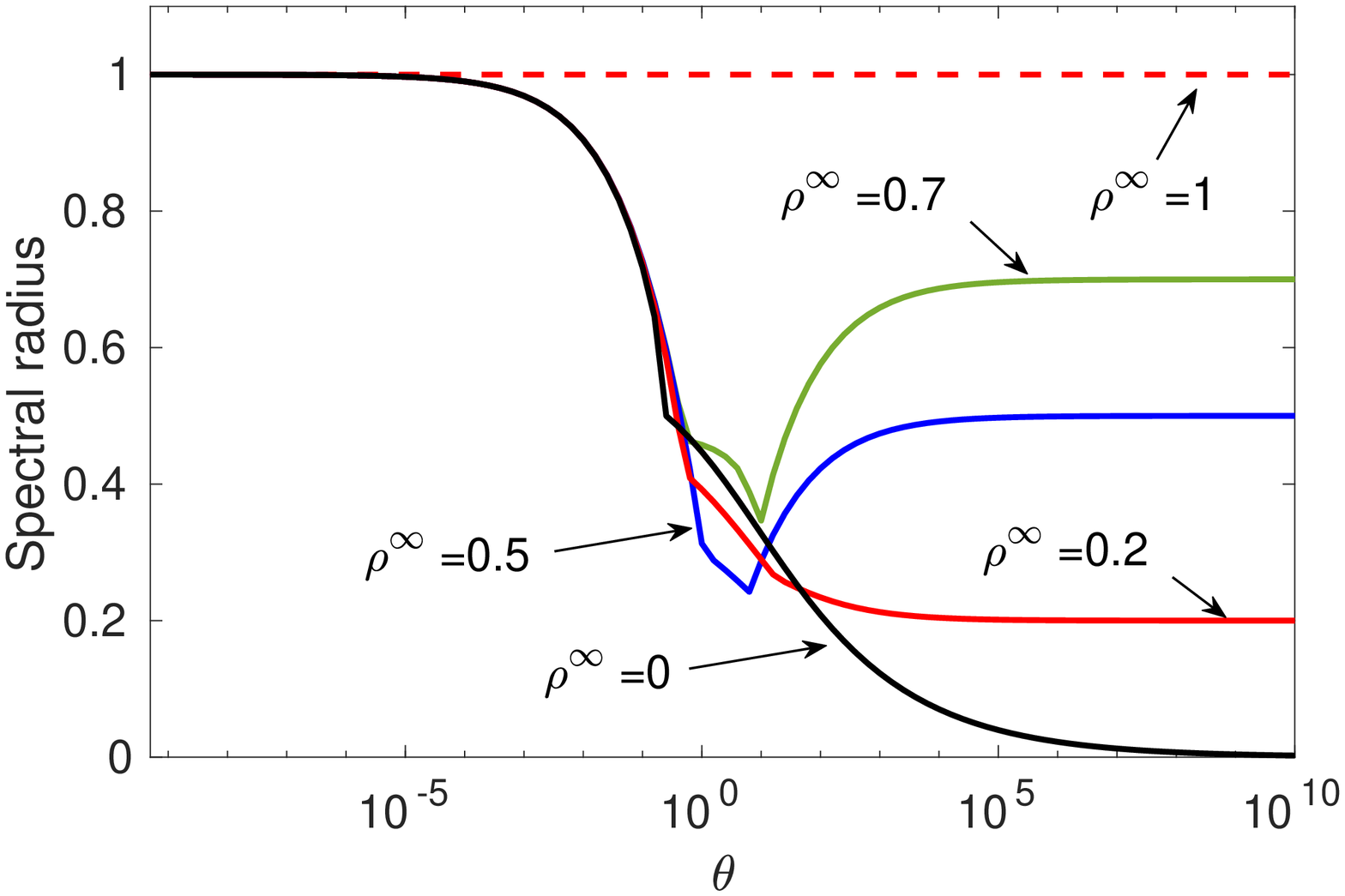} }
  \subfigure[Fifth-order method ]{\centering\includegraphics[width=7.15cm]{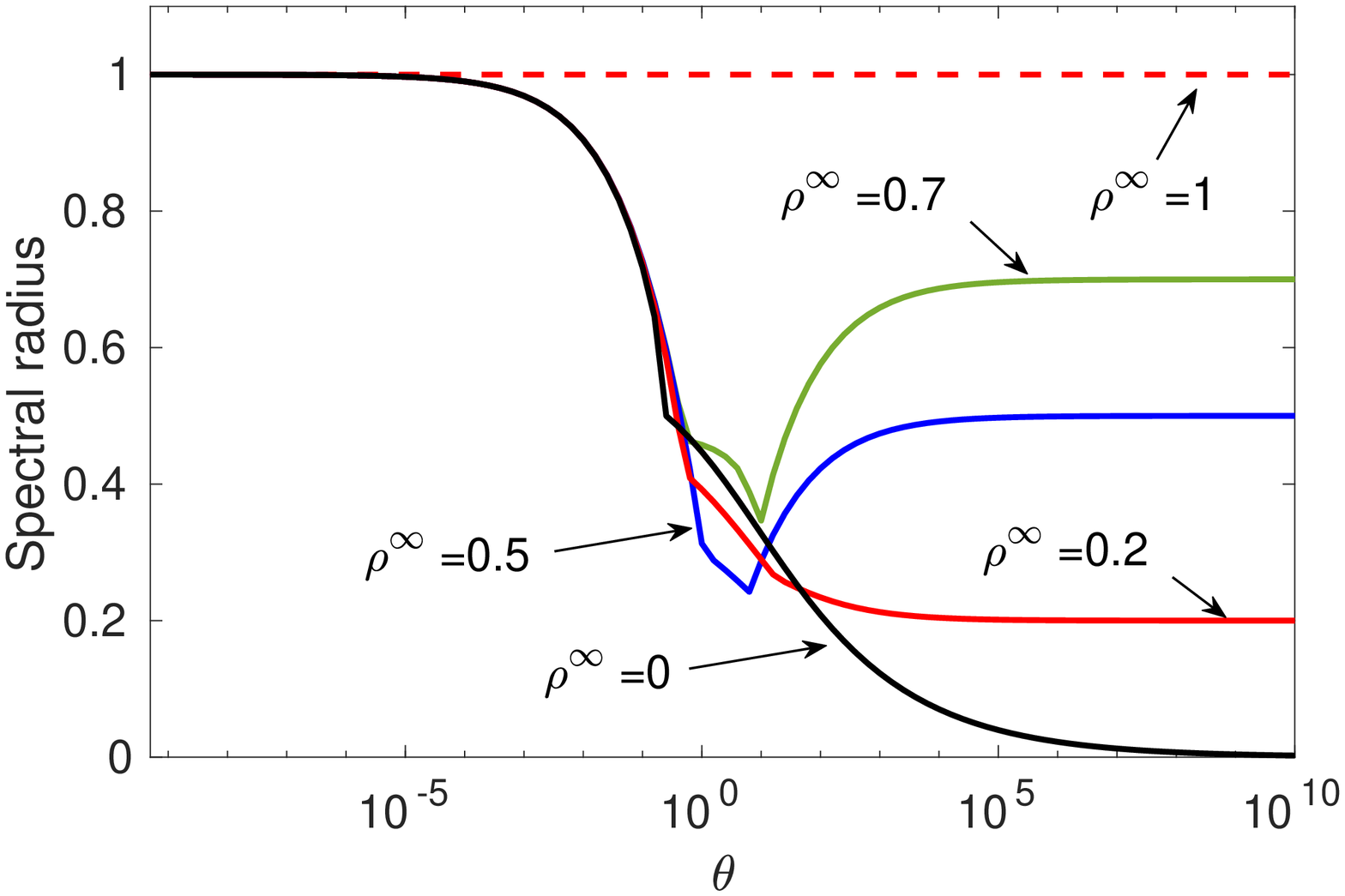} }
  \subfigure[Sixth-order method ]{\centering\includegraphics[width=7.15cm]{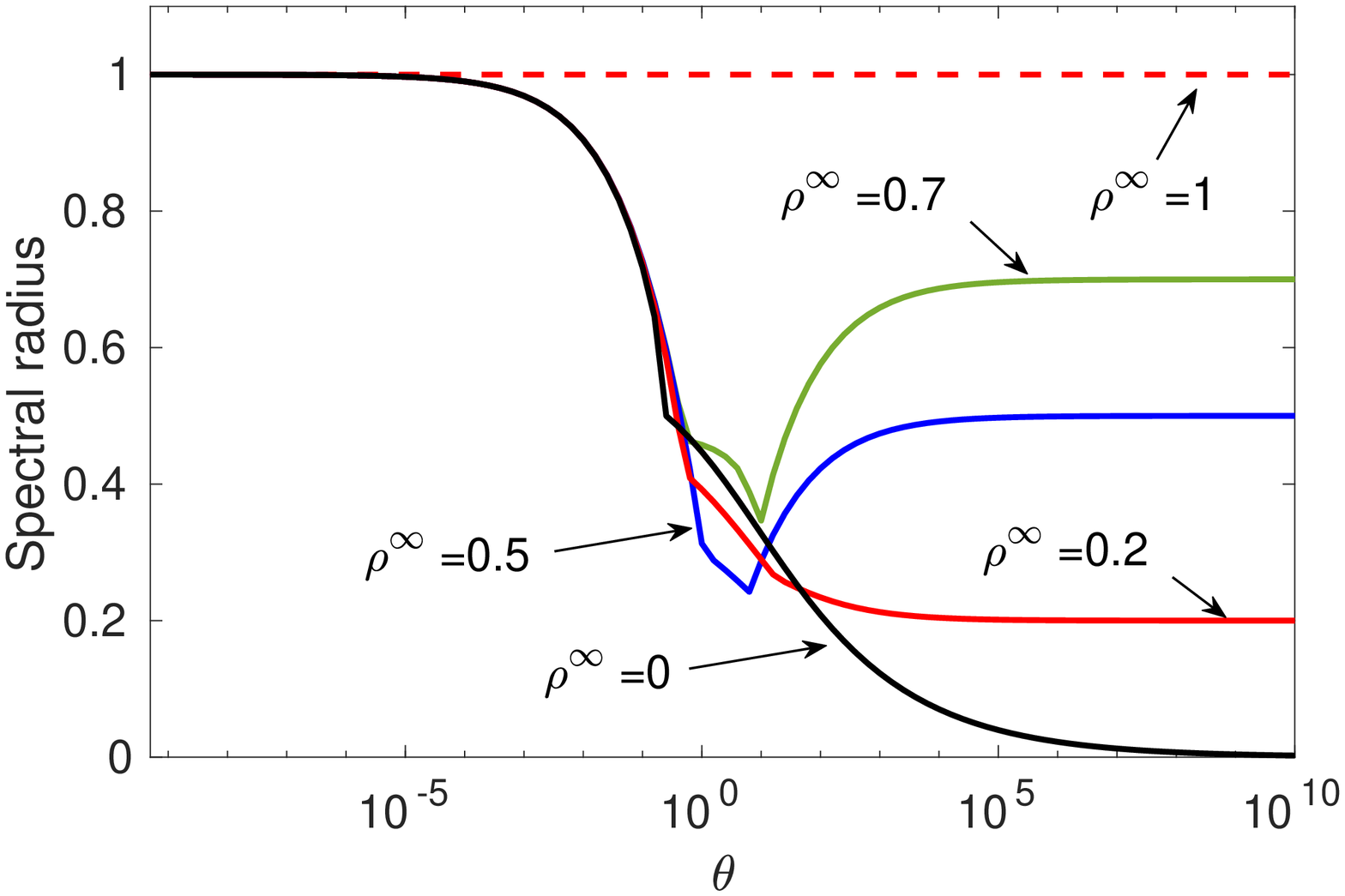} }
  \caption{{Spectral radius of our one-parameter family of methods ($\rho^\infty_1=\rho^\infty_2=\cdots=\rho^\infty_k=\rho^\infty$).}}
  \label{fig:spec}
\end{figure}

\subsection{A-Stability of the method}

We now investigate our method to solve stiff systems with complex entries (i.e., advection problems). The second Dahlquist barrier states that the stable region of a multistep method for a stiff equation shrinks for accuracy orders higher than two (see,~\cite{ hairer2010solving}). Our approach delivers the amplification matrix of~\eqref{eq:HOampblock}, which requires solving $k$ systems that are form identical to the second-order generalized-$\alpha$ method. Each block decouples from the others; therefore, their eigenvalues are independent as well. Thus, we solve $k$ independent systems that lead to a high-order method with an invariant stability region. While, the analysis in~\ref{sec:stab} supports our claims, herein, we consider a problem with complex eigenvalues, $\lambda_\theta \in \mathbb{C}$. Similarly, we consider the amplification matrix~\eqref{eq:HOampblock} to obtain the region of stability as:
\begin{equation}\label{eq:comp}
  \left\{\theta \in \mathbb{C}:\left| \frac{\alpha_k+(\alpha_f-1) \gamma_k \theta}{\alpha_k+\alpha_f \gamma_k \theta}\right|\leq 1,\,\, \left| \frac{\alpha_j }{\alpha_j+ \gamma_j \theta}\right|\leq 1  \right\} ,\qquad j=1,\cdots,k-1.
\end{equation}
Considering~\eqref{eq:cond}, one can show that imposing $Re(\theta)\geq0$ is sufficient to satisfy~\eqref{eq:comp}, which proves that our method is A-stable. 

\begin{remark}
  Setting $k=1$ and $\rho^\infty=1$, in~\eqref{eq:comp} defines a method with the trapezoidal method's stability region, which has a second-order of accuracy and A-stability.
\end{remark}

\begin{theorem}{L-stability:}
  The method introduced in~\eqref{eq:ho}-\eqref{eq:ho3} shows high-order L-stability for $\rho_\infty=0$. 
  \begin{proof}
    Given that the method is L-stable (following this section's analysis), recall that $\tilde{\lambda}$ is the eigenvalue of the amplification matrix~\eqref{eq:HOampblock}, thus, we only require to prove~\cite{ hairer2010solving}:
    \begin{equation}
      \theta \to \pm\infty \implies \tilde{\lambda}\to 0.
    \end{equation}
    For $\theta \to +\infty$, we already show the spectral behavior in~\eqref{eq:eigen}. For $\theta \to -\infty$, we can show that we obtain similar eigenvalues. Therefore,  redefining the parameters by setting  $\tilde{\lambda}=\rho^\infty$, completes our proof. 
  \end{proof}
\end{theorem}

Figure~\ref{fig:a-stab} shows that the accuracy order and the stability region are independent; the figure shows the boundedness of the system's eigenvalues for a problem with complex eigenvalues. We see that the system's spectral radius behavior is similar for the second, third, and fifth-order accuracy orders. 

\begin{figure}[h!]
  \subfigure[Second order, $\rho^\infty=0$  ]{\centering
    \includegraphics[width=0.32\textwidth]{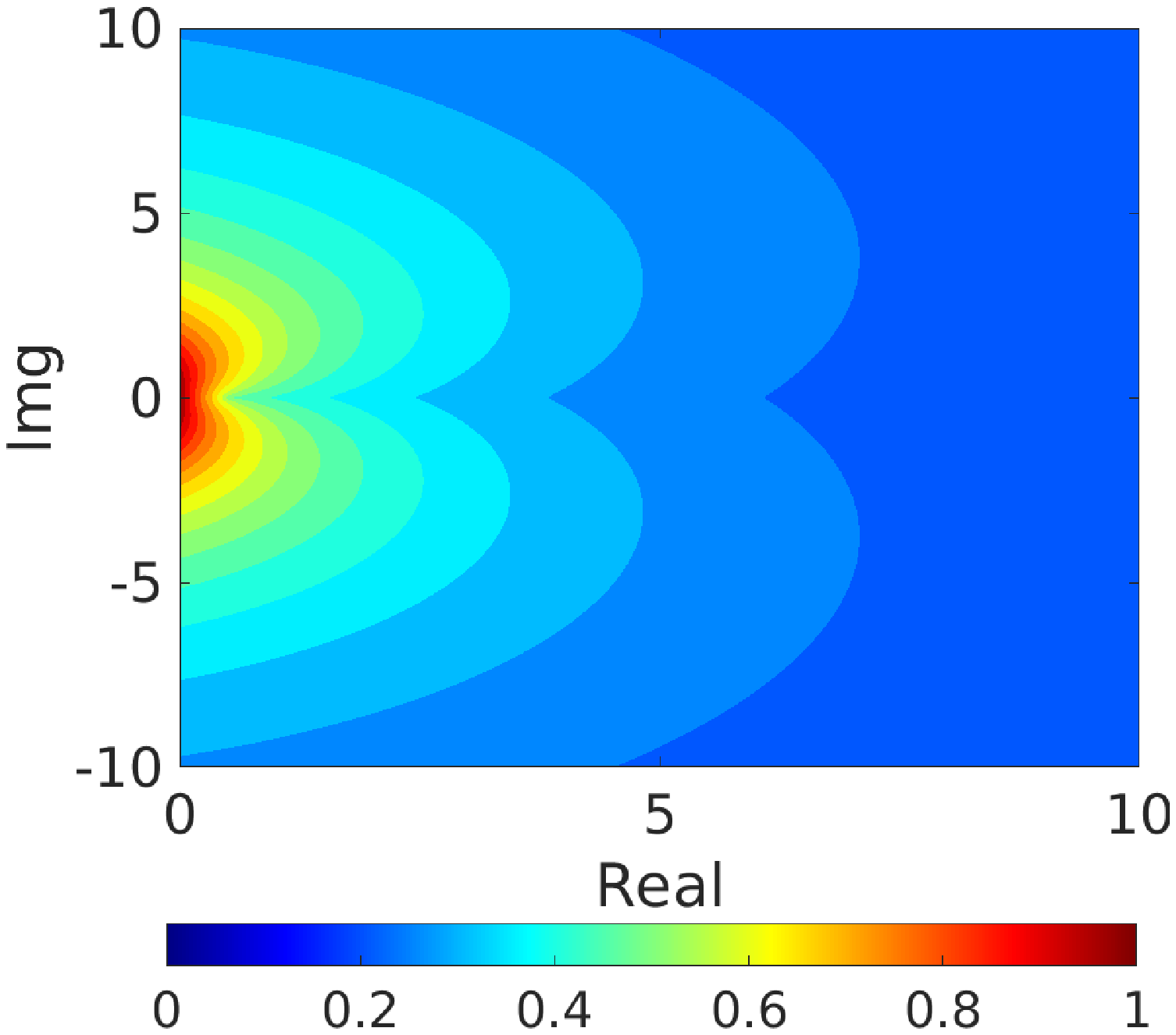}}
  \subfigure[Third order, $\rho^\infty=0$  ]{\centering
    \includegraphics[width=0.32\textwidth]{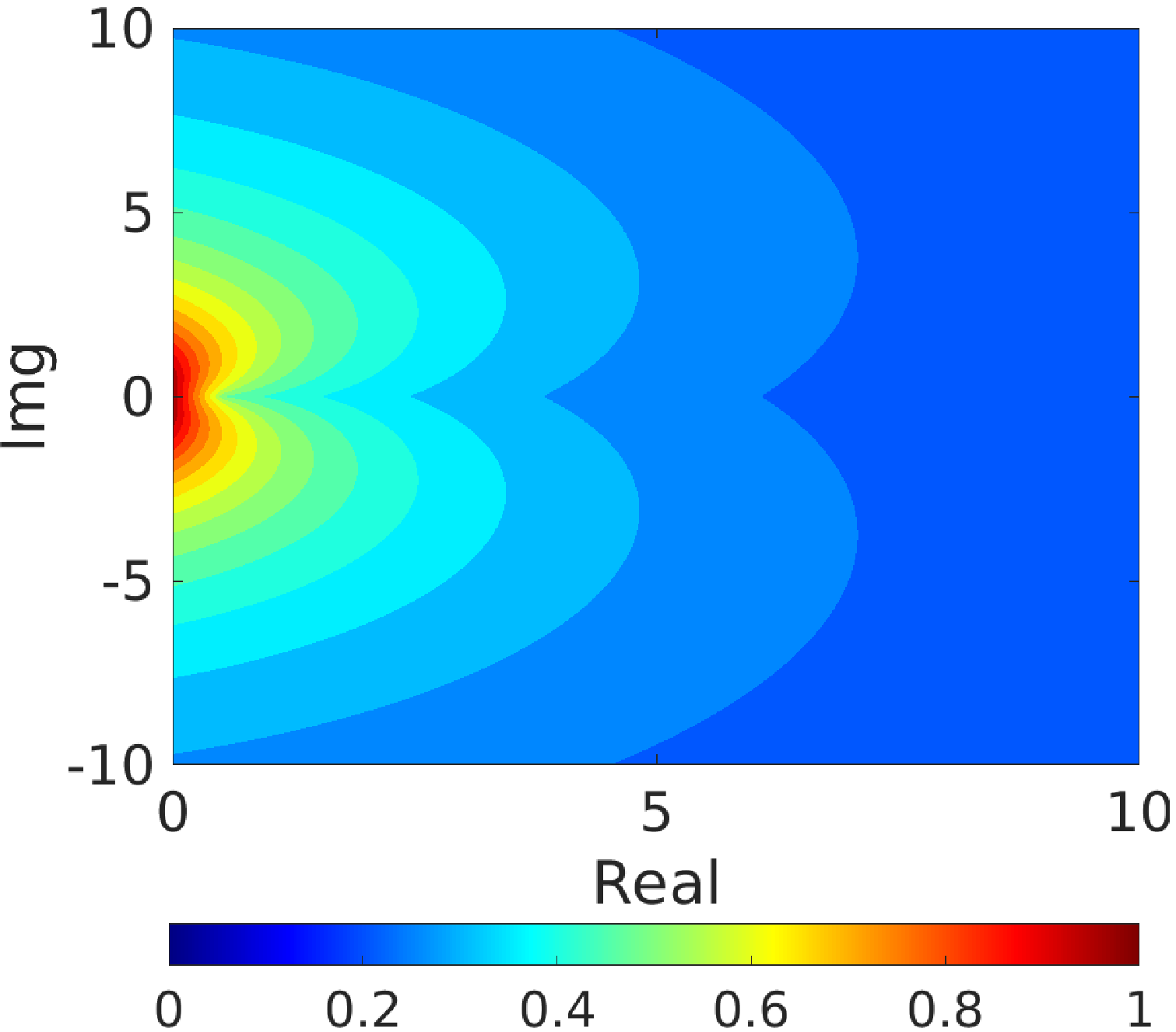}}
  \subfigure[Fifth order, $\rho^\infty=0$]{\centering
    \includegraphics[width=0.32\textwidth]{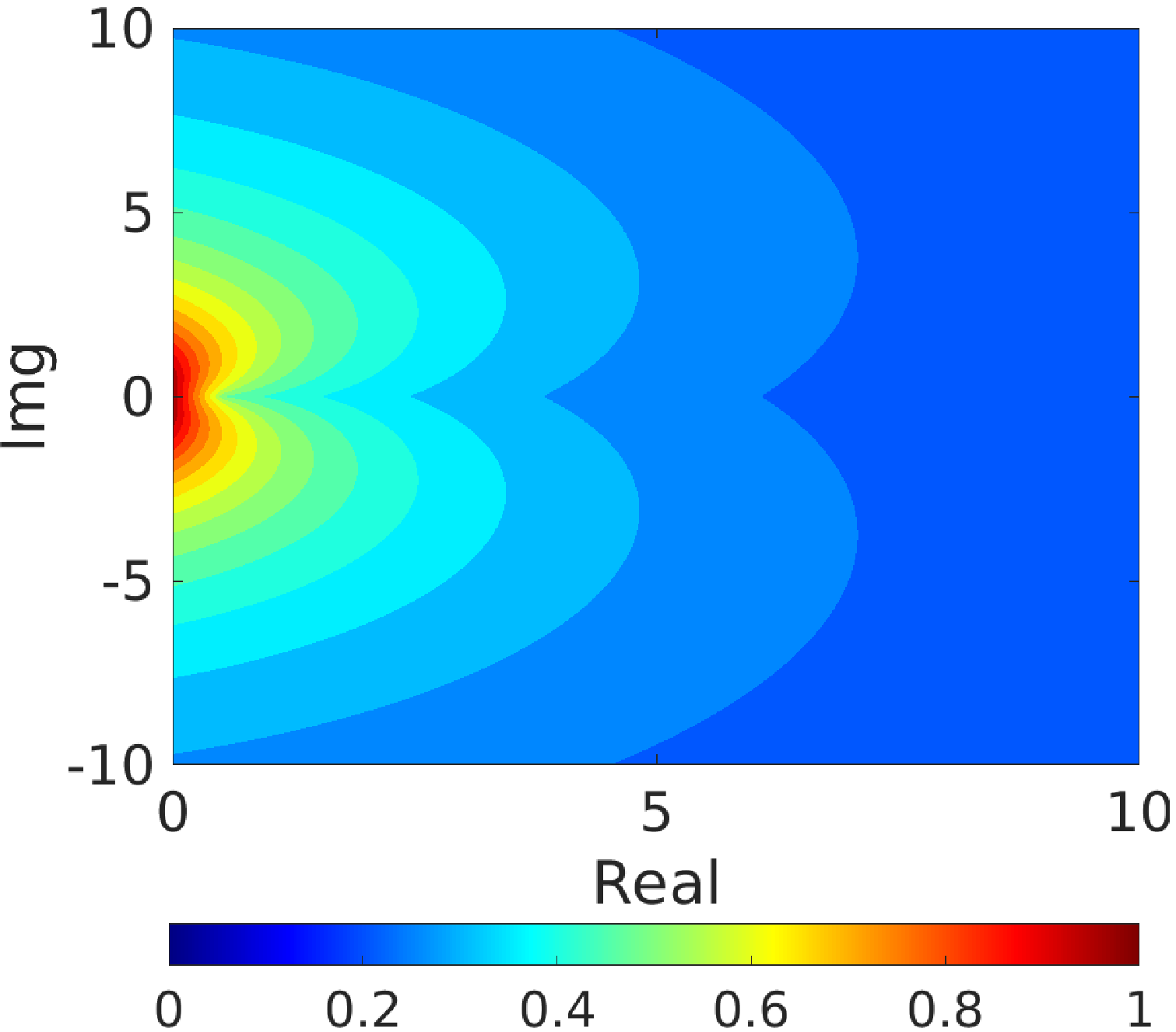}}
  \subfigure[ Second order, $\rho^\infty=0.5$ ]{\centering
    \includegraphics[width=0.325\textwidth]{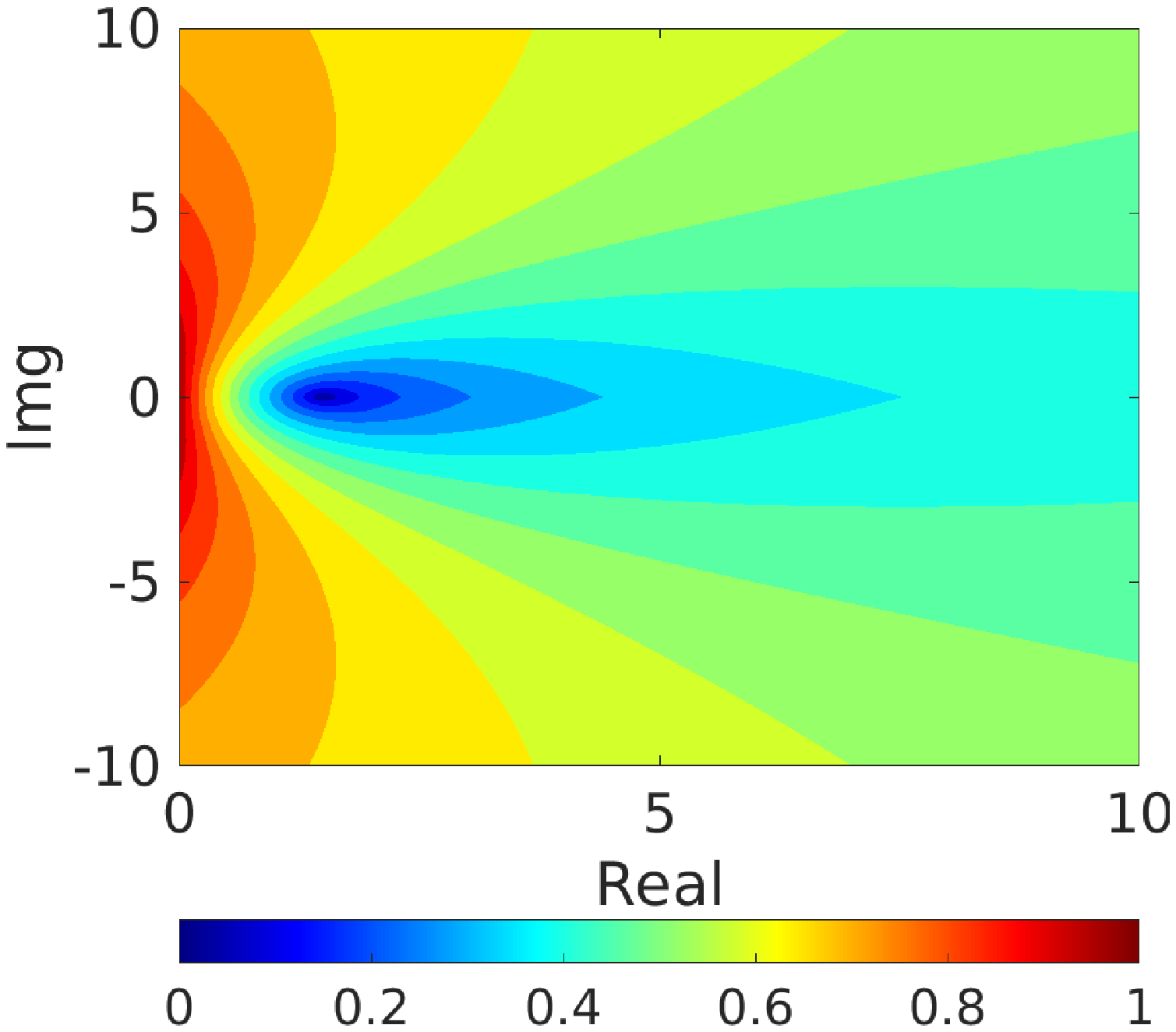}}
  \subfigure[ Third order, $\rho^\infty=0.5$]{\centering
    \includegraphics[width=0.325\textwidth]{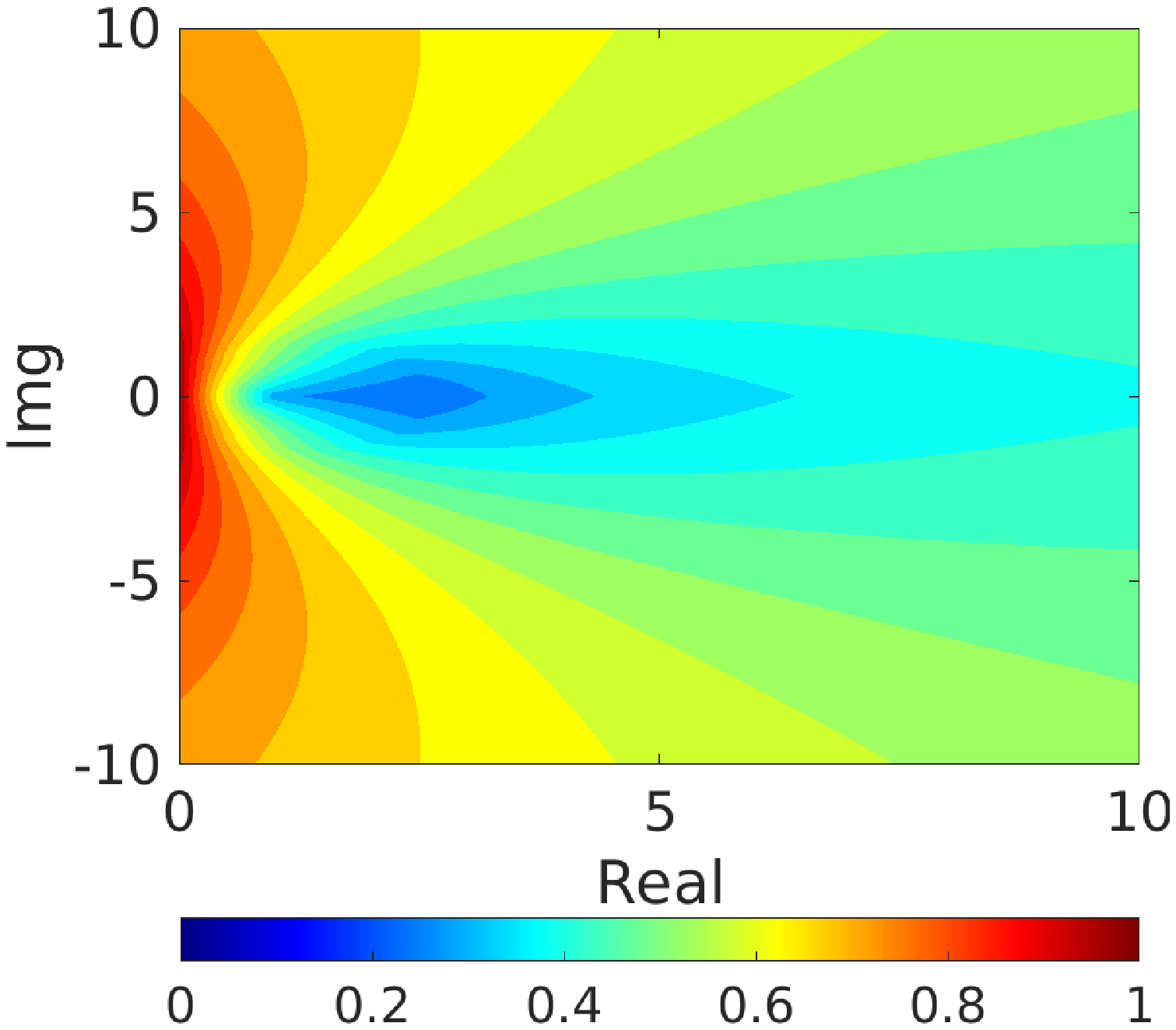}}
  \subfigure[ Fifth order, $\rho^\infty=0.5$]{
    \includegraphics[width=0.325\textwidth]{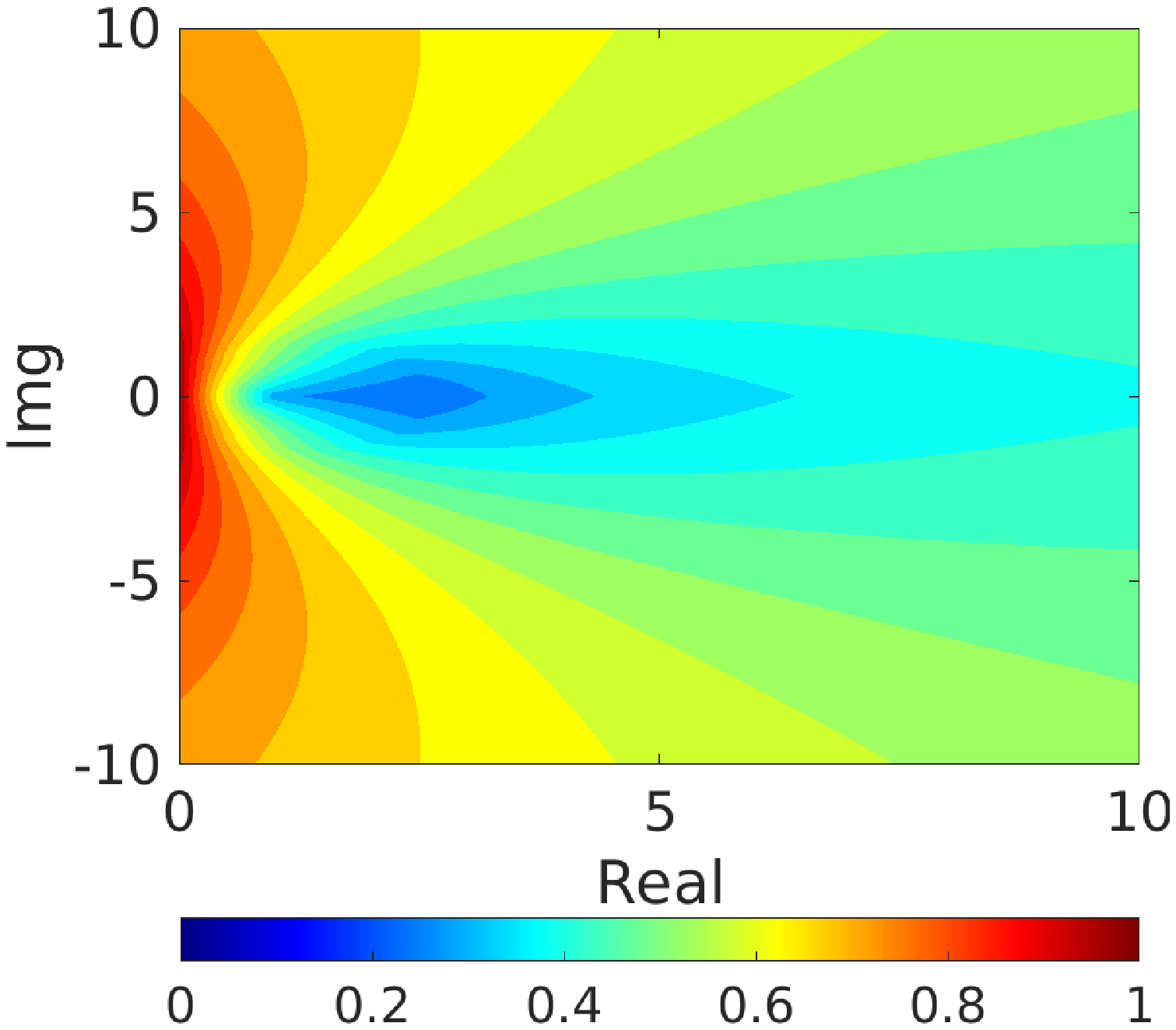}}
  \caption{\label{fig:a-stab} Invariant stability behaviour (system's spectral radius unchanged) by increasing the method's accuracy (from left to right: original generalized-$\alpha$  (second-order), third-order, and fifth-order accuracy).}
\end{figure}

\section{Concluding remarks} \label{sec:Conc}

We derive and analyze a new class of higher-order generalized-$\alpha$ methods for solving parabolic problems that maintain all the attractive features of the original (second-order) generalized-$\alpha$ method. In particular, at each time step, we obtain $(3/2k)^{th}$ and $(3/2k+1/2)^{th}$ order of accuracy in time, respectively, for even and odd $k$, by solving $k$ matrix systems consecutively and implicitly. We then update the other $k$ variables explicitly. We derive a one-parameter family with dissipation control using a user-specified parameter $\rho^\infty$. Our method is A-stable for arbitrarily high accuracy. Furthermore, setting $\rho^\infty=0$, our method shows L-stability behaviour.  

\section*{Acknowledgement}

This publication was also made possible in part by the CSIRO Professorial Chair in Computational Geoscience at Curtin University and the Deep Earth Imaging Enterprise Future Science Platforms of the Commonwealth Scientific Industrial Research Organisation, CSIRO, of Australia. This project has received funding from the European Union's Horizon 2020 research and innovation programme under the Marie Sklodowska-Curie grant agreement No 777778 (MATHROCKS). The Curtin Corrosion Centre and the Curtin Institute for Computation kindly provide ongoing support.
The authors also would like to acknowledge the contribution of an Australian Government Research Training Program Scholarship in supporting this research.
	
\section*{References}
\bibliographystyle{elsarticle-harv}\biboptions{square,sort,comma,numbers}
\bibliography{ref}

\appendix
\section{The coefficients of Cayley--Hamilton Theorem}\label{app:1}

For a given invertible matrix $G$, we can determine the coefficients $c_0=(-1)^n\det(G)$ and other coefficients $c_i, \, i\in \{{1,\cdots,n-1}\}$  in terms complete exponential Bell polynomials $B_l$ as~\cite{ horn1990matrix}:
\begin{equation}
  c_{n-l}=\frac{(-1)^l}{l!}B_l\left(s_1,\,-s_2\,,\,2!s_3\,,\cdots,\,(-1)^{l-1}(l-1)!s_l\,\right),
\end{equation} 
where $s_l$ is the power sum of symmetric polynomials of the eigenvalues:
\begin{equation}\label{eq:s}
  s_l =\sum_{i=1}^{l}\lambda^l_i=tr \left(G^l\right),
\end{equation}
with $tr \left(G^l\right)$ the trace of $G^l$. The $l^{th}$ complete exponential Bell polynomial reads:
\begin{equation}
  B_l(x_1,\cdots,x_l)=\sum_{m=1}^{l}B_{l,m}(x_1,\cdots,x_{l-m+1}),
\end{equation}
and defining the partial exponential Bell polynomials $B_{l,m}$ as:
\begin{equation}
  B_{l,m}(x_1,\cdots,x_{l-m+1})=\sum \frac{l!}{j_1!j_2!\cdots j_{l-m+1}!}\left(x_1\right)^{j_1}\left(\frac{x_1}{2}\right)^{j_2}\left(\frac{x_{l-m+1}}{(l-m+1)!}\right)^{j_{(l-m+1)}}
\end{equation}
where the sum is taken over all sequences $j_1, j_2, j_3, \cdots, j_{l-m+1}$ of non-negative integers such that these two conditions are satisfied: 
\begin{equation}
  \begin{aligned}
    j_1+j_2+\cdots +j_{l-m+1}=m,\\
    j_1+2j_2+\cdots +(l-m+1)j_{l-m+1}=l.
  \end{aligned}
\end{equation}
Following~\cite{ householder2013theory}, it is also possible to determine $B_l$ using the determinant as:
\begin{equation}
  B_l(x_1,\cdots,x_l)=\det 
  \begin{bmatrix}
    x_1&x_2&\dfrac{x_3}{2!}& \cdots &\dfrac{x_l}{(l-1)!}\\
    -1&x_1&x_2& \cdots &\dfrac{x_{l-1}}{(l-2)!}\\
    0&-2&x_1& \cdots &\dfrac{x_{l-2}}{(l-3)!}\\
    0&0&-3& \cdots &\dfrac{x_{l-3}}{(l-4)!}\\
    \vdots&&\vdots&&\vdots\\
    0&0&\cdots &-(l-1)&x_1
  \end{bmatrix}.
\end{equation}
For example, one can readily obtain
\begin{align}
  B_2(x_1,x_2)&=x_1^2+x_2,\label{eq:B2}\\
  B_3(x_1,x_2,x_3)&=x_1^3+3x_1x_2+x_3,\\
  B_4(x_1,x_2,x_3,x_4)&=x_1^4+6x_1^2x_2+4x_1x_3+3x_2^2+x_4,\label{eq:B4}\\
  B_5(x_1,x_2,x_3,x_4,x_5)&=x_1^5+10x_1^3x_2+15x_2^2x_1+10x_1^2x_3\\
              &+10x_3x_2+5x_4x_1+x_5,\\
  B_6(x_1,x_2,x_3,x_4,x_5,x_6)&=x_1^6+15x_1^4x_2+20x_1^3x_3+45x_1^2x_2^2\\
              &+15x_2^3+60x_3x_2x_1+15x_1^2x_4+10x_3^2+15x_4x_2+6x_5x_1+x_6.
\end{align}
Then, we using~\eqref{eq:s} in~\eqref{eq:B2} and multiplying by $U^{n-1}$, we obtain~\eqref{eq:a40}. Similarly, introducing~\eqref{eq:s} into~\eqref{eq:B4} and multiplying by $U^{n-3}$ leads to~\eqref{eq:C4}. 

\end{document}